\documentclass[a4paper,12pt]{article}
\usepackage[ngerman, english]{babel}
\usepackage[T1]{fontenc}
\usepackage[utf8]{inputenc}
\usepackage{amsmath}
\usepackage{amssymb}
\usepackage{amsthm}
\usepackage{graphicx}
\usepackage{here}
\usepackage{color}
\usepackage{xcolor}
\usepackage{enumerate}
\usepackage{lmodern}
\usepackage{fancyvrb}
\usepackage[plainpages=false]{hyperref}
\usepackage{caption}
\usepackage{subfigure}
\usepackage{epstopdf}
\newcommand{\ca}{\mathcal}

\newtheorem{defi}{Definition}[section]
\newtheorem{theo}[defi]{Theorem}

\newtheorem{lem}[defi]{Lemma}
\newtheorem{con}[defi]{Conjecture}

\newtheorem{cor}[defi]{Corollary}

\theoremstyle{remark}

\title{Some conjectures on $r$-graphs and equivalences}

\author{Yulai Ma$^1$ , Eckhard Steffen$^2$, Isaak H.~Wolf$^2$, 
Junxue Zhang$^3$ \\
	\footnotesize
	$^1$ Center for Combinatorics and LPMC, Nankai University, Tianjin 300071,  China\\
	\footnotesize
	$^2$ Department of Mathematics, Paderborn University, Warburger Str.\ 100, 33098 Paderborn,
	Germany\\
	\footnotesize
	$^3$ School of Mathematics and Statistics, Beijing Institute of Technology, Beijing, China
	\\ \footnotesize yulai.ma@upb.de, es@upb.de, isaak.wolf@upb.de, jxuezhang@163.com}

\date{}

\begin{document}
	
	\maketitle
	
	\begin{abstract}
		An $r$-regular graph is an $r$-graph, if every odd set of vertices is connected to its complement by at least $r$ edges. Seymour [On multicolourings of cubic graphs, and conjectures of Fulkerson and Tutte.~\emph{Proc.~London Math.~Soc.}~(3), 38(3): 423-460, 1979] conjectured (1) that every planar
		$r$-graph is $r$-edge colorable and (2) that every $r$-graph has $2r$ perfect matchings such that every edge is contained in precisely two of them. We study several variants of these conjectures. 
		
		A $(t,r)$-PM is a multiset of $t \cdot r$ perfect matchings of an $r$-graph $G$ such that every edge
	is in precisely $t$ of them. We show that 
		the following statements are equivalent for every $t, r \geq 1$: 
		
		\begin{enumerate} 
			\item Every planar $r$-graph has a $(t,r)$-PM.
			\item Every $K_5$-minor-free $r$-graph has a $(t,r)$-PM.
			\item Every $K_{3,3}$-minor-free $r$-graph has a $(t,r)$-PM. 
            \item Every $r$-graph whose underlying simple graph has crossing number at most $1$ has a $(t,r)$-PM. 
		\end{enumerate}
\end{abstract}
	

	\section{Introduction}
	All graphs considered in this paper are finite and may have parallel edges 
	but no loops. A graph is simple if it contains no loops and no parallel edges. The underlying simple graph of a graph $G$ is denoted by $G_s$. 

The vertex set of a graph $G$ is denoted by $V(G)$ and its edge set by $E(G)$. 
The number of parallel edges joining two vertices $u,v$ of $G$ is denoted by $\mu_G(u,v)$ and $\mu(G)=\max\{\mu_G(u,v)\colon u,v\in V(G)\}$. For two disjoint sets $X,Y\subseteq V(G)$ we denote by $E_G(X,Y)$ the set of edges with one end in $X$ and the other end in $Y$. If $Y= X^c$ we denote $E_G(X,Y)$ by $\partial_G(X)$.  In particular, when $X=\{x\}$,  we write $E_G(x,Y)$ instead of  $E_G(\{x\},Y)$.
The graph induced by $X$ is denoted by $ G[X] $. 
The number of components of a graph $G$ is denoted by $\omega(G)$.
A set of $k$ vertices, whose removal increases the number of components of a graph, is called a {\em $k$-vertex-cut}.
A graph $G$ on at least $k+1$ vertices is called {\em $k$-connected} if $G-X$ is connected for every $X\subset V(G)$ with $|X|\leq k-1$. Moreover, the graph obtained from $G$ by identifying all vertices in $X$ to a single vertex $w_X$ and deleting all resulting loops is denoted by $G/X$. If $G[X]$ is connected, we say $G/X$ is obtained from $G$ by contracting $X$.
	
A graph $H$ is a \emph{minor} of the graph $G$ if $H$ can be obtained from a subgraph of $G$ by contracting edges. Moreover, a graph $G$ is \emph{$H$-minor-free} if $H$ is not a minor of $G$. 
A graph is \emph{$r$-regular} if every vertex has degree $r$. 
An $r$-regular graph $G$ is an \emph{$r$-graph}, if $|\partial_G(X)| \geq r$ for every $X \subseteq V(G)$ of odd cardinality.  Note that  all $r$-graphs are of even order.
	
	Let $G$ be a graph and $S$ be a set. An \emph{edge-coloring} of $G$ is a mapping $f\colon E(G)\to S$. It is a 
	\emph{$k$-edge-coloring} if $|S|=k$, and it is \emph{proper} 
	if $f(e) \not = f(e')$ for any two adjacent edges $e$ and $e'$. 
	The smallest integer $k$ for which $G$ admits a proper $k$-edge-coloring 
	is the \emph{edge-chromatic number} of $G$, which is denoted by $\chi'(G)$. A \emph{matching} is a set $M\subseteq E(G)$ such that no two edges of $M$ are adjacent. Moreover, $M$ is \emph{perfect} if every vertex of $G$ is incident with an edge of $M$.
	If $\chi'(G)$ equals the maximum degree of $G$, then $G$ is \emph{class $1$}; otherwise $G$ is \emph{class $2$}. 
The \emph{crossing number} of a graph $G$ is the minimum number of crossings in a drawing of $G$ in the plane and it is denoted by $cr(G)$.

	Definitions which are not given can be found in \cite{west2021combinatorial}.
Tait \cite{Tait_1880} proved that the 4-Color-Theorem is equivalent to the statement 
that every planar 3-graph is class 1. Jaeger \cite{jaeger1980tait} showed that this statement can be extended to 3-graphs with crossing number at most $1$.

	\begin{theo} [\cite{jaeger1980tait}] \label{Thm: Jaeger_crosN 1}
        The following statements are equivalent.
		\begin{enumerate}
			\item Every planar $3$-graph is class 1. 
			\item Every $3$-graph with crossing number at most $1$ is class $1$. 
		\end{enumerate}    
	\end{theo}
	
	Tutte's  4-flow conjecture \cite{tutte1966algebraic} states that every 
	bridgeless graph without Petersen-minor has a nowhere-zero 4-flow. For cubic graphs this conjecture is equivalent to the statement that every bridgeless 
	cubic graph without Petersen-minor is class 1. 
 As the Petersen graph has crossing number $2$,
	Theorem \ref{Thm: Jaeger_crosN 1} can be seen as a first approximation 
	to a proof of Tutte's conjecture for cubic graphs. 
	Tutte's 4-flow conjecture for cubic graphs is announced to be proved, see \cite{edwards2016three, robertson1997tutte, robertson2019excluded, thomas1999recent}.
	Furthermore, every bridgeless cubic graph with girth at least $6$ has a 
	Petersen-minor \cite{robertson2019girth}. 
	Seymour stated the following conjecture for planar $r$-graphs.

	\begin{con}[\cite{seymour1979multi,seymour1979unsolved}] \label{Conj: Seymour exact}
		For $r \geq 1$, every planar $r$-graph is class $1$. 
	\end{con}

The following conjecture, which naturally generalizes the edge-coloring version 
 of Tutte's $4$-flow conjecture for cubic graphs to $r$-graphs, is attributed to Seymour in \cite{dvovrak2016packing}.

	\begin{con} \label{Conj: Petersen minor}
		Let $G$ be an $r$-graph. If $G$ has no Petersen-minor, then $G$ is class 1. 
	\end{con}

Already in 1984, Ellingham \cite{ellingham1984petersen} 
studied specific Petersen-minor-free regular graphs.

	\begin{theo}[\cite{ellingham1984petersen}] \label{Thm: Ellingham84}
		Let $r \geq 3$ and $G$ be an $r$-regular graph with a 2-factor consisting of two
		circuits $A$ and $B$, both chordless in $G$. If $G$ has no Petersen-minor, then $G$ is class $1$. 
	\end{theo}
	
If $G$ is an $r$-graph of class $1$, then every color class of an
 $r$-edge-coloring is a perfect matching of $G$. It seems to be that some multiples 
 of $r$-graphs are class $1$. Indeed, it is even conjectured that it is sufficient to replace every edge by two parallel edges.
	
	\begin{con} [\cite{seymour1979multi}] \label{Con: gen Berge Fulkerson}
		For $r \geq 1$, every $r$-graph has $2r$ perfect matchings such that every edge is in precisely two of them.
	\end{con}

%From the proofs for $r \leq 8$ of Conjecture \ref{Conj: Seymour exact} it follows that	
%Conjecture \ref{Con: gen Berge Fulkerson} is proved for $r \leq 8$ for planar $r$-graphs and it is open for $ r \geq 9$. For general $r$-graphs, Conjecture \ref{Con: gen Berge Fulkerson} is open for all $r \geq 3$. Indeed, it is even an open question whether there exists an integer $t \geq 1$ such that every (planar) $r$-graph has $t \cdot r$ perfect matchings such that every edge is in precisely $t$ of them. 

We are going to combine these conjectures.  
A $( t,r )$-PM of an $ r $-graph $ G $ is a multiset of $ t \cdot r $ perfect matchings $ M_1,\ldots,M_{tr} $ of $ G $ such that every edge of $ G $ is contained in exactly $ t $ of  them. Trivially, if 
$G$ is class 1, then it has a $(t,r)$-PM for every $t \geq 1$.
The following two conjectures are weak versions of Conjecture \ref{Con: gen Berge Fulkerson}.

\begin{con} \label{Conj: gen BF-Conj weak - strong}
 There is a $t \geq 1$ such that for all $r \geq 1$ every $r$-graph has a 
 $(t,r)$-PM.   
\end{con}

\begin{con} \label{Conj: gen BF-Conj weak - weak}
For every $r \geq 1$ there is a $t_r \geq 1$ such that every $r$-graph has a $(t_r,r)$-PM.   
\end{con}

Clearly, Conjecture \ref{Con: gen Berge Fulkerson} implies Conjecture \ref{Conj: gen BF-Conj weak - strong}, which implies Conjecture
\ref{Conj: gen BF-Conj weak - weak}. 
The reductions of these two conjectures to planar $r$-graphs or 
to Petersen-minor-free $r$-graphs are also open.

\begin{con} \label{Conj: combination1}
1. There is $t \geq 1$ such that for all $r \geq 1$ every planar $r$-graph has a $(t,r)$-PM.\\
2. For every $r \geq 1$ there exists $t_r \geq 1$ such that every 
    planar $r$-graph has a $(t_r,r)$-PM.
\end{con}

\begin{con} \label{Conj: combination2}
    1. There is $t \geq 1$ such that for all $r \geq 1$ every Petersen-minor-free $r$-graph has a $(t,r)$-PM.\\
    2. For every $r \geq 1$ there exists $t_r \geq 1$ such that every 
    Petersen-minor-free $r$-graph has a $(t_r,r)$-PM.
\end{con}

Clearly, if Conjecture \ref{Conj: Seymour exact} is true, then 
both statements of 
Conjecture \ref{Conj: combination1} are true for every $r \geq 1$. 
Similarly, if Conjecture \ref{Conj: Petersen minor} is true, then
both statements of Conjecture \ref{Conj: combination2} are true for 
every $r \geq 1$.

Every $K_4$-minor-free graph is planar and 
Conjecture \ref{Conj: Seymour exact} is confirmed for $K_4$-minor-free
$r$-graphs by Seymour \cite{seymour1990colouring}. Natural subsets of the set of Petersen-minor-free $r$-graphs are the sets of $K_5$-minor-free and of 
$K_{3,3}$-minor-free $r$-graphs. We will prove that the restriction of 
Conjecture \ref{Conj: combination2} to these sets of $r$-graphs is equivalent to Conjecture \ref{Conj: combination1}. 
The following theorem is the main result of this paper. 
	
\begin{theo}\label{Theorem-equiv-minor-free}
		For any $t \geq 1$ and $r \geq 1$, the following statements are equivalent.
		\begin{enumerate}
			\item Every planar $r$-graph has a $( t,r )$-PM.
			\item  Every $K_5$-minor-free $r$-graph has a $( t,r )$-PM.
			\item  Every $ K_{3,3} $-minor-free $r$-graph has a $( t,r )$-PM.
            \item Every $r$-graph $G$ with $cr(G_s) \leq 1$ has a $( t,r )$-PM.
		\end{enumerate}
	\end{theo}
	
Conjecture \ref{Conj: Seymour exact} is proved for $r \leq 8$ in a sequel 
of papers
	\cite{7chudnovsky2015edge, 8chudnovsky2015edge, dvovrak2016packing, guenin2003packing}. 
Thus, for $r \leq 8$, every $r$-graph with no $K_5$-minor or with no $K_{3,3}$-minor or with underlying simple graph with crossing number at most 1 has a 
$(t,r)$-PM for every $t \geq 1$. 

We also show that if $H$ is a minimum counterexample to one of the aforementioned conjectures, then it is $3$-connected. 

\section{Basic results}
	
\subsection{Clique-sums}
%	\begin{defi}
%		Let $ G $ and $ H $ be simple graphs which contain a clique of the same size. The {\bf clique-sum} of $ G $ and $ H $, denoted by $G\bigoplus H $, is an operation that forms a new simple graph obtained from their disjoint union by identifying a clique of $ G $ and one of $ H $ with the same size and possibly deleting some edges in the clique. A {\bf $ k $-clique-sum}, denoted by $ \bigoplus_k $, is a clique-sum in which both cliques have exactly $ k $ vertices. The set of all identifying vertices in the resulting graph is called the vertex set with respect to the clique-sum.
%	\end{defi}

If $H$ is a simple graph with induced subgraphs $H_1$ and $H_2$ and $S \subseteq V(H)$ 
such that 
$V(H) = V(H_1) \cup V(H_2)$, $E(H) = E(H_1) \cup E(H_2)$, and 
$S = V(H_1) \cap V(H_2)$, then we say that $H$ is obtained from $H_1$ and $H_2$
by {\em pasting} together these graphs along $S$.

Let $k \geq 1$ be an integer. A graph $G$ is a $k$-{\em clique-sum} of $G_1$ and $G_2$, denoted by $G_1 \bigoplus_k G_2$, if $G$ is obtained by pasting
together two graphs $G_1'$ and $G_2'$ along a set $S'$ of order $k$ and for 
each $i \in \{1,2\}$, $G_i$ is obtained from
$G_i'$ by extending $G_i'[S']$ to a complete graph of order $k$. 
%If the order of $S'$ is irrelevant we say that $G$ is the {\em clique-sum} of $G_1$ and $G_2$
%and denote it by $G_1 \bigoplus G_2$.
The set $S'$ is called the vertex set with respect to the clique-sum.

Informally, the clique-sum of two graphs can be considered as gluing together two graphs along a clique and removing some edges of the clique.
We also say that the clique-sum is an operation that yields $G$ from $G_1$ and $G_2$
(which contain a clique of the same order).

\begin{lem} \label{Lemma-cut-vertex}
 Let $k \geq 1$ be an integer and $G$ be a connected graph. 
 If $G = G_1\bigoplus_{k}G_2$, then each of the following statements
 holds: 
  \begin{enumerate}
      \item $G_i$ is connected for each $i \in \{1,2\}$.
      %Every $G_i$ is connected for  $ i\in \{1,2\}$.
      \item For every $S \subseteq V(G_i) \colon 
      \omega(G_i-S) \leq \omega(G-S)$.
	%\item Every vertex-cut $S_i$ of $ G_i $ corresponds to a vertex-cut $S$ of $ G $. In particular, $|S_i|=|S|$ and $\omega(G_i-S_i) \leq \omega(G-S)$.\\
	\item If $G_i\ncong K_k$ for any $i\in \{1,2\}$, then the vertex set with respect to the clique-sum $ \bigoplus_{ k}$ is a vertex-cut of $ G $.
  \end{enumerate}
	\end{lem}
	
	\begin{proof} Statement 3 is trivial, so we just prove statements 1 and 2 below.
		
		Statement 1: Suppose to the contrary that $G_1$ is disconnected. By definition, the clique-sum operation $\bigoplus_k$ must be taken on a component of $G_1$, say $G_{1}^1$, and $G_2$. As a result, $G$ is a disjoint union of $G_1-V(G_{1}^1)$ and $G_{1}^1\bigoplus_k G_2$, which implies that $G$ is disconnected. This is a contradiction.
		
		%(I) Suppose $G_1$ is disconnected. Let $G_{11}$ be a component of $G_1$ and $G_{12}=G_1-G_{11}$. Without loss of generality, we assume that $G_{11}$ contains a $k$-cique. Then $G=G_1\bigoplus_k G_2=G_{12}\cup (G_{11}\bigoplus_k G_2)$, which contradicts the fact that $G$ is connected. Then $G_i$ is connected for any $i\in [2]$.
		
		Statement 2: %Let $S$ be the vertex set in $G$ corresponding to $S_i$. Trivially, $|S|=|S_i|$. 
		%By symmetry, we just argue for the case $i=1$. 
		Assume that $G_i-S$ contains components $G_{i}^1,\ldots,G_{i}^\ell$, where $\ell=\omega(G_i-S)\geq 1$. Without loss of generality, we assume that the clique-sum operation $\bigoplus_k$ is taken on  $G_i[V(G_{i}^1)\cup S]$ and $G_{3-i}$.   Consequently, $G_{i}^j$ is also a component of $G-S$ for each $j\in\{2,\ldots,\ell\}$. Since 
  $G-(S \cup \bigcup\limits_{j=2}^{\ell} V(G_{i}^j))$ is non-empty, we obtain $\omega(G_i-S) \leq \omega(G-S)$.
		%
		%Let $F$ be the $k$-clique of $G$ obtained from identifying  $F_1$ and $F_2$, where $F_i$ is a  $k$-clique of $G_i$.
		%For convenience, let $F=F_1=F_2$. 
		%
		%(II) 
		%Let $G_{i1}, G_{i2}, \ldots, G_{it}$ be the components of $G_i-S_i$. We have $G-S_i=(\cup_{1\le j\le t}G_{ij} )\cup G_2'$, where $G_2'=G_2-F_2$. We can see $e(G[V(G_{ij}),V(G_{i\ell})])=0$ for any $j,\ell\in [t]$ and $j\ne \ell$. 
		%We assert that there is at most one components $G_{ij}$ of $G_{i1}, G_{i2},\ldots, G_{it}$ such that $e(G[V(G_{ij}), V(G_2')])\ne 0$. 
		%Suppose $e(G[V(G_{ij}), V(G_2')])\ne 0$ and $e(G[V(G_{i\ell}), V(G_2')])\ne 0$, where $j,\ell\in [t]$ and $j\ne \ell$. 
		%Then $F\cap G_{ij}\ne \emptyset$ and $F\cap G_{i\ell}\ne \emptyset$. Since $F$ is  a clique, $e(G[V(G_{ij}), V(G_{i\ell})])\ne 0$, a contradiction. Thus  $\omega(G_i-S_i)\le \omega(G-S_i)$.
		%
		%(III) We have $G_i-F_i\ne \emptyset$ for any $i\in [2]$. Since $G_1$ and $G_2$ are disjoint,  $e(G[V(G_1-F_1), V(G_2-F_2)])=0$. Hence, $F$ is  a vertex-cut of $G$.
	\end{proof}

	%\begin{cor} \label{Coro-geneal-cut-vertex}
	%	Let $ G $ be a connected graph obtained from $G_1,\ldots, G_k $ by a sequence of  clique-sums $  \bigoplus^1, \ldots, \bigoplus^{k-1}$.
	%	Each of the following holds.\\
	%	(I) Every $ G_i$ is connected for  $ i\in\{1,\ldots,k\}$.\\
	%	(II) For each  $ i\in\{1,\ldots,k\}$ and $ j\in\{1,\ldots,k-1\} $, any vertex-cut of $ G_i $ and any vertex set with respect to $ \bigoplus^j$  is a vertex-cut of $ G $.
	%\end{cor}
	
	\subsection{$K_5$-minor-free or $K_{3,3}$-minor-free graphs}
	
	\begin{theo}  [Wagner's Theorem \cite{wagner1937eigenschaft}, see also \cite{thomas1999recent}]\label{Theorem-Wagner-character} Let $ G $ be a simple graph.
 \begin{enumerate}
     \item $ G $ is planar if and only if $ G $ is $ K_5 $-minor-free and $ K_{3,3} $-minor-free.
	\item $ G $ is $ K_5 $-minor-free if and only if $ G $ can be obtained from planar graphs and the Wagner graph $V_8$ by  means of $0$-, $1$-, $2$-, $3$-clique-sums, where the Wagner graph is the $3$-regular graph as shown in Figure \ref{fig:Wagner graph}.
	\item $ G $ is $ K_{3,3} $-minor-free if and only if $ G $ can be obtained from planar graphs and $ K_5 $ by means of $0$-, $1$-, $2$-clique-sums.
  \end{enumerate}
	\end{theo}

 \begin{figure}[htbp]
     \centering
     \includegraphics[width=0.5\linewidth]{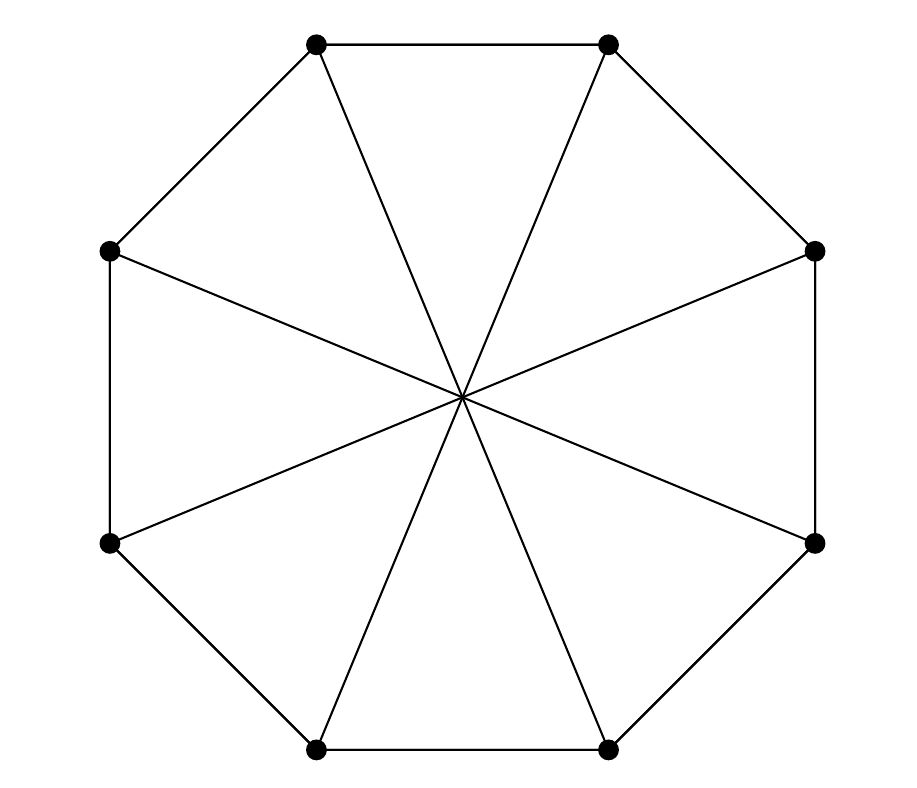}
     \caption{The Wagner graph $V_8$}
     \label{fig:Wagner graph}
 \end{figure}
 
	\begin{cor}[\cite{wagner1937eigenschaft}]\label{Coro-planarity-minor-free}
		Each of the following holds.
  \begin{enumerate}
    \item Every $ 4 $-connected $ K_5 $-minor-free simple graph is planar.
    \item Every $ 3 $-connected $ K_{3,3} $-minor-free simple graph except $K_5$ is planar.
  \end{enumerate}
	\end{cor}
	
%	\begin{lem}\label{Lemma-3con-3components} 
%		Every $3$-connected non-planar $ K_5 $-minor-free simple graph $ G $ except the Wagner graph has a vertex-cut $S=\{u,v,w\}$ of size  $ 3 $ such that $G-S$ contains at least three components and the simple graph $G'$ with $V(G')=V(G)$ and $E(G')=E(G)\cup \{uv,vw,wu\}$ is $ K_5 $-minor-free.
%\end{lem}

\begin{lem}\label{Lemma-3con-3components} 
Let $G$ be $3$-connected non-planar $K_5 $-minor-free simple graph.
If $ G \not = V_8$, then $G$ has a $3$-vertex-cut $S=\{u,v,w\}$ such that 
$\omega(G-S) \geq 3$ and the simple graph $G'$ with $V(G')=V(G)$ and $E(G')=E(G)\cup \{uv,vw,wu\}$ is $ K_5 $-minor-free.
	\end{lem}

 	\begin{proof}
		Since $G$ is $3$-connected, it follows by  Lemma \ref{Lemma-cut-vertex} and Statement 2 of Theorem \ref{Theorem-Wagner-character}, that $G$ is obtained from connected planar graphs $G_1,\ldots, G_k$ by a sequence $\sigma$ of $3$-clique-sums. 
  Note that $k \geq 2$, since $G$ is not planar. Without loss of generality, we assume that $k$ is minimum and that the first clique-sum operation in the sequence $\sigma$ is taken on graphs $G_1$ and $G_2$. Denote $H=G_1\bigoplus_3 G_2$, and let $S$ be the vertex set with respect to $\bigoplus_3$. Note that $H$ is not planar since $k$ is minimum.  As a consequence, $|V(G_1)|,|V(G_2)|>3$. For every $i \in \{1,2\}$, if $G_i-S$ is connected, then $G_i$ can be embedded in the plane such that the boundary of the outer face is given by the triangle $G_i[S]$. Therefore, if both $G_1-S$ and $G_2-S$ are connected, then $H$ is planar, a contradiction. Thus, without lose of generality we assume $\omega(G_1-S) \geq 2$, which implies $\omega(H-S)\geq 3$. With Lemma~\ref{Lemma-cut-vertex} we deduce $\omega(G-S)\geq 3$.
  
  Furthermore, by the definition of clique-sums and Statement 2 of Theorem \ref{Theorem-Wagner-character}, the graph obtained from $G$ by extending $G[S]$ to a triangle is still $K_5$-minor-free. This completes the proof.
	\end{proof}

% \begin{figure}[htbp]
%     \centering
%     \includegraphics[width=0.3\linewidth]{K5-e.pdf}
%     \caption{The planar embedding of $K_5-e$}
%     \label{fig:K5-e}
% \end{figure}

	\subsection{$r$-graphs}
 
	Let $G$ be an $r$-graph. An edge-cut $\partial(X)$ is a \emph{non-trivial tight edge-cut} if $\vert X \vert$ is odd, 
 $\vert \partial(X) \vert = r$ and $\vert X \vert, \vert V(G) \setminus X \vert > 1$. 
 
	\begin{lem}\label{Lemma-2-vertex-cut}
		For each $ r\geq 3 $, if a connected $ r $-graph $ G $ has a $2$-vertex-cut, then either $ G $ has a non-trivial tight edge-cut or $ G_s $ is a circuit of length $ 4 $.
	\end{lem}
	
	\begin{proof}
		By the definition of an $r$-graph, $G$ has no cut-vertex, and hence $G$ is $2$-connected. Assume that $\{u,v \}$ is a vertex-cut  of $ G $, and $ G-\{u,v\} $ has $ t $ components, denoted by $ G_1,\ldots, G_t $, where $ t\geq 2 $.  
		For convenience, let $ |E_G(V(G_i),u)|=a_i $ and $ |E_G(V(G_i),v)|=b_i $ for each $ i\in\{1,\ldots,t\} $.  Since $ G $ is $ 2 $-connected,  we have $a_i\neq 0$ and $ b_i\neq0$.  Note that the order of $ G $ is even, so the proof can be split into two cases as follows.
		
		{\bf Case 1.} $ G-\{u,v\} $ contains at least  two components of  odd order, say $ G_1$ and $ G_2$.
		
		By definition of $ r $-graphs, $ a_1+b_1= |\partial_G(V(G_1))|\geq r $ and $ a_2+b_2= |\partial_G(V(G_2))|\geq r $.  Moreover, $a_1+a_2\leq d_G(u)=r $ and $b_1+b_2\leq d_G(v)=r $. So $ a_1=b_2 $, $ a_2=b_1$, $ a_1+b_1=r $, and  $ a_2+b_2=r $.  It follows that $ t=2 $ and $ u $ is not adjacent to $ v $. If $ |V(G_i)|\geq 2 $ for some $ i\in\{1,2\} $, then $ \partial_G(V(G_i)) $ is a non-trivial tight edge-cut.  If $ |V(G_1)|=|V(G_2)|=1 $, then $ G_s $ is a circuit of length $ 4 $.
		
		{\bf Case 2.} The order of $ G_i $ is even for  each $ i\in\{1,\ldots,t\} $.
		
		By definition of $ r $-graphs, $b_i+(r-a_i)=|\partial_G(V(G_i)\cup\{u\})|\geq r $ and  $a_i+(r-b_i)=|\partial_G(V(G_i)\cup\{v\})|\geq r $. This implies $ a_i=b_i $. Thus,  $\partial_G(V(G_i)\cup\{u\}) $ is a non-trivial tight edge-cut.
	\end{proof}

	\begin{lem}\label{Lemma-3-vertex-cut}
		For each $ r\geq 3 $, if a $ 3 $-connected $ r $-graph $ G $ has a $3$-vertex-cut $S$ such that $G-S$ contains at least three components of odd order,  then either $ G $ has a non-trivial tight edge-cut or $ G_s \cong K_{3,3}$  .
	\end{lem}
	
	\begin{proof}
		Assume that $G_1,G_2,G_3$ are three odd  components  of $G-S$.  Note that $|\partial_G(S)|\geq |\partial_G(V(G_1))|+ |\partial_G(V(G_2))|+ |\partial_G(V(G_3))|$. By definition of $r$-graphs, we have $|\partial_G(S)|\leq 3r$, and $|\partial_G(V(G_i))|\geq r$ for each $i\in\{1,2,3\}$. Hence, $|\partial_G(S)|= 3r$ and $|\partial_G(V(G_i))|= r$. This implies that
		$S$ is an independent set and $G-S$ contains exactly three components. Therefore, if $G_s\ncong K_{3,3}$, then
		$|V(G_i)|\geq3$ for some $i\in\{1,2,3\}$, and so $\partial_G(V(G_i))$ is a non-trivial tight edge-cut. 
	\end{proof}
	
	\begin{theo}[K\"onig's Theorem \cite{konig1931graphs}]
 \label{Theorem-bipartite-class1}
		Every $ r $-regular bipartite graph is of class $ 1 $.
	\end{theo}

In \cite{ma2023sets} the smallest $r$-graphs of class $2$ are characterized for every $r \geq 3$. In particular, they are of order $10$. Thus we obtain the following corollary.

	\begin{cor} \label{Theorem-V8-class1}
		Every $ r $-graph with the Wagner graph as the underlying graph is of class $ 1 $.
	\end{cor}
 
	%\begin{obs}
	%Let $ G $ be an $ r $-graph with a  non-trivial tight edge-cut $ \partial_G(X) $. Each of the following holds.\\
	%(I) If both  $G/X$ and $G/(V(G)\setminus X)$ are of class $ 1 $, then $ G $ is of class $ 1 $.\\
	%(II) If the Generalized Berge-Fulkerson Conjecture holds on  both  $G/X$ and $G/(V(G)\setminus X)$, then it also holds on $ G $ .
	%\end{obs}
	
	%\begin{lem}\label{Lemma-tight-cut-PM}
		%Let $ G $ be an $ r $-graph with a  non-trivial tight edge-cut $ \partial_G(X) $. If  both  $G/X$ and $G/(V(G)\setminus X)$ have a $( t,r )$-PM, then $ G $ has a $( t,r )$-PM.
	%\end{lem}
	
	%\begin{proof}
		%For convenience,  set $\partial_G(X)=\{e_1,e_2,\ldots,e_r\}$. In addition, let $e_i'$ and $e_i''$ be the corresponding edge of $e_i$ in $G/X$ and $G/(V(G)\setminus X)$, respectively. Moreover, denote a $(t,r)$-PM of $G/X$  by $\{M'_1,\ldots, M'_{tr}\}$ and  a $(t,r)$-PM of $G/(V(G)\setminus X)$ by $\{M''_1,\ldots, M''_{tr}\}$.
		%For each $e_i'$ ($e_i''$, respectively), there are exactly $t$ perfect matchings of $\{M'_1,\ldots, M'_{tr}\}$ ($\{M''_1,\ldots, M''_{tr}\}$, respectively) containing it, denoted by $M'_{i1}, \ldots, M'_{it}$ ($M''_{i1}, \ldots, M''_{it}$, respectively). Let $M_{(i-1)t+j}=(M_{ij}'\cup M_{ij}''\cup \{e_i\})\setminus \{e_i',e_i''\}$ for each $i\in \{1,\ldots,r\}$ and $j\in \{1,\ldots,t\}$. As a result,
		% $\{M_1, M_2, \ldots, M_{tr}\}$ is a $(t,r)$-PM  of $G$. 
	%\end{proof}
 %\iw{this proof is notation-wise a bit ugly}

Let $G$ be a graph and let $x \in V(G)$ with $|N_G(x)| \geq 2$.
A lifting (of $G$) at $x$ is the following
operation: Choose two distinct neighbors $y$ and $z$ of $x$, delete an edge $e_1$ joining $x$ with $y$,
delete an edge $e_2$  joining $x$ with $z$ and add a new edge $e$  joining $y$ with $z$.
	
	\begin{theo}[\cite{ma2023sets}]\label{theo:r-graph_lifting}
		Let $r\geq 2$ be an integer, let $G$ be a connected $r$-graph and let $X$ be a non-empty proper subset of $V(G)$. If $\vert X \vert$ is even, then $G/X$ can be transformed into a connected $r$-graph by applying $\frac{1}{2}\left \vert \partial_G(X)\right|$ lifting operations at $w_X$ and deleting $w_X$, where $w_X$ is the contracting vertex in $G/X$. If $\vert X \vert$ is odd, then $G/X$ can be transformed into a connected $r$-graph by applying $\frac{1}{2}\left( \vert \partial_G(X) \vert - r \right)$ lifting operations at $w_X$.
	\end{theo}
	%Let $ \mathcal{G}$ be a set of $ r $-graphs such that $ G/e\in \mathcal{G}$  for any graph $ G\in \mathcal{G}$ and any edge $ e\in E(G) $. Let $ \mathcal{G}_2$ be the subset of $ \mathcal{G}$ consisting of all graphs of class $ 2 $.
	
	%\begin{lem}
	%If $\mathcal{G}_2\neq\emptyset$, then a  graph in $\mathcal{G}_2$ with minimum order has  no non-trivial tight edge-cut.
	%\end{lem}
	
	%\begin{proof}
	% By way of contradiction, suppose that $ G $ is an $ r $-graph in $\mathcal{G}_2$  with minimum order, which contains  a non-trivial tight edge-cut $ \partial_G(X) $, where $ X\subset V(G) $.  Noth that  both $ G[X] $ and $ G[V(G)\setminus X] $ are connected. Hence, two $ r $-graphs obtained from $ G $ by contracting $ X $ and $ V(G)\setminus X $, respectively, are in $ \mathcal{G}$, which implies that they are of class $ 1 $. Since $|\partial_G(X)|  =r$, their $ r $-colorings can be combined to an $ r$-coloring of $ G $. This is a contradiction to that $ G $ is of class $ 2 $.
	%\end{proof}

We finish this section with the following statement on the structure of possible minimum counterexamples to aforementioned conjectures.  

\begin{theo}\label{Thm: min counterexample} 
    If $H$ is a possible minimum counterexample to one 
    of the statements of Conjectures \ref{Conj: Seymour exact} - \ref{Conj: combination2} or to any of statements 1, 2, or 3 of Theorem \ref{Theorem-equiv-minor-free}, then $H$  is a 3-connected 
    $r$-graph and it does not contain
    any non-trivial tight edge-cut. 
\end{theo}

\begin{proof}
We prove the statement for Conjecture \ref{Conj: gen BF-Conj weak - weak}.
There is $r \geq 1$ such that for all $t \geq 1$ there is an $r$-graph $G$
which does not have a $(t,r)$-PM. Let $t$ be arbitrary but fixed and let 
$H$ be a minimum counterexample with respect to $t$.

Suppose to the contrary that $H$ has a non-trivial tight edge-cut $\partial_H(S)$.
Then $H[S]$ and $H[V(H)\setminus S]$ are connected. Thus,
$H/S$ and $H/(V(H)\setminus S)$ are $r$-graphs which both have a 
$(t,r)$-PM, which can be combined to a $(t,r)$-PM of $H$, 
a contradiction. 

An $r$-graph is 2-connected and suppose that $H$ has a 2-vertex-cut. 
It follows with Lemma \ref{Lemma-2-vertex-cut} that $H_s$ is a circuit of length $4$. But then $H$ has a $(t,r)$-PM, since $H_s$ is regular and class 1. 

For Conjecture \ref{Conj: gen BF-Conj weak - strong} fix $r$ for which the conjecture is
supposed to be false and let $H$ be a minimum possible counterexample. The statement follows as above.   

The statements for the other conjectures follow analogously, too. Note that 
$H[S]$ and $H[V(H)\setminus S]$ are connected and thus, they have the 
same excluded minors as $H$.  
\end{proof}

\section{Proof of Theorem \ref{Theorem-equiv-minor-free}} 

Trivially, each of statements $2$, $3$, and $4$ implies statement $1$. So it remains to prove that statement  $ 1 $ implies statement $2$, $3$, and $4$.
	
\subsection*{Statement $1$ implies statement $2$}
By way of contradiction, we suppose that $ G $ is a minimal counterexample to statement 2 with respect to $|V(G)|+|E(G)|$.
By Theorem \ref{Thm: min counterexample}, $G$ has no non-trivial tight edge-cut and $G$ is 3-connected. 
	
%\begin{claim}\label{Claim-2-conn}
%	$G$ is $2$-connected.
%\end{claim}
	%\begin{proof}
		%If $G$ is disconnected then each component of $G$ is a $K_{5}$-minor-free $r$-graph. By the minimality of $G$, we apply statement $2$ on each component, and then obtain a $(t,r)$-PM of $G$, which is a contradiction. Hence, $G$ is connected. Indeed, $G$ is $2$-connected since a connected $r$-graph  does not contain a cut vertex.
	%\end{proof}
	
	%\begin{claim}\label{Claim-no-tight-cut}
		%$G$ has no non-trivial tight edge-cut. In particular, $G$ is $3$-connected.
	%\end{claim}
	%\begin{proof}
		%Suppose to the contrary that $\partial_G(X)$ is a non-trivial tight edge-cut of $G$. Since $G$ is an $r$-graph, both $G[X]$ and $G[V(G)\setminus X]$ are connected by Claim \ref{Claim-2-conn}. Therefore, $G/X$ and $G/(V(G)\setminus X)$ are both $K_{5}$-minor-free. Due to the minimality of $G$, each of $G/X$ and $G/(V(G)\setminus X)$ has a $( t,r )$-PM. Thus, $ G $ has a $( t,r )$-PM by Lemma \ref{Lemma-tight-cut-PM}, a contradiction since $G$ is a counterexample.
		
		%Now we show that $G$ is $3$-connected. If not, $G$ has a $2$-vertex-cut. By Lemma \ref{Lemma-2-vertex-cut}, $G_s$ is a circuit of length $4$. As a result, $G$ has a $( t,r )$-PM by Theorem \ref{Theorem-bipartite-class1}, a contradiction again.
	%\end{proof}
	
	By statement 1 and Corollary \ref{Theorem-V8-class1}, the underlying graph $G_s$ of $G$ is neither a planar graph nor the Wagner graph. It is worth noting that $G_s$ is $K_5$-minor-free. 
	Therefore, by Lemma \ref{Lemma-3con-3components}, $G_s$ has a $3$-vertex-cut $S$ such that $G_s-S$ contains at least three components, and so does $G$. We may assume $S=\{u,v,w\}$. According to Lemma~\ref{Lemma-3-vertex-cut} and Theorem~\ref{Theorem-bipartite-class1},  $G-\{u,v,w\}$ contains exactly one component of odd order, and hence it contains at least two components of even order, denoted by $G_1$ and $G_2$. By Theorem \ref{theo:r-graph_lifting}, for each $i \in \{1,2\}$ there are integers $d_i$, $h_i$ and $k_i$ such that a new $r$-graph $G_i'$ can be obtained from $G-V(G_i)$ by adding $d_i$ edges  joining $u$ and $v$, adding $h_i$ edges  joining $v$ and $w$, and adding $k_i$ edges  joining $w$ and $u$. For each $i\in\{1,2\}$, let $|E_G(u,V(G_i))|=a_i$, $|E_G(v,V(G_i))|=b_i$, $|E_G(w,V(G_i))|=c_i$. We have $a_i=d_i+k_i$, $b_i=d_i+h_i$ and $c_i=h_i+k_i$, which imply
  \begin{align*}
		d_i= \frac{1}{2} \left(a_i+b_i-c_i\right), \quad
		h_i= \frac{1}{2} \left(-a_i+b_i+c_i\right), \quad
		k_i= \frac{1}{2} \left(a_i-b_i+c_i\right).
	\end{align*}
	 According to Lemma \ref{Lemma-3con-3components}, each $G'_i$ is $K_5$-minor-free. Thus, by the minimality of $G$, each $G'_i$ has a $(t,r)$-PM $\mathcal{M}_i=\{M^i_1,\ldots,M^i_{tr}\}$. For every $j \in \{1,\ldots,tr\}$, let $\Tilde{M}^2_{j}=M^2_{j} \cap (E(G_1) \cup \partial_{G}(V(G_1)))$ and set $\Tilde{\mathcal{M}}_2=\{\Tilde{M}^2_1,\ldots,\Tilde{M}^2_{tr}\}$. Every element of $\Tilde{\mathcal{M}}_2$ contains exactly zero or two edges of $\partial_{G}(V(G_1))$ and hence, there are exactly $td_1$ matchings in $\Tilde{\mathcal{M}}_2$ that saturate $u,v$, exactly $th_1$ matchings in $\Tilde{\mathcal{M}}_2$ that saturate $v,w$ and exactly $tk_1$ matchings in $\Tilde{\mathcal{M}}_2$ that saturate $w,u$. As a consequence, $\mathcal{M}_1$ can be transformed to an $(t,r)$-PM of $G$ by appropriately replacing the edges that had been added to $G-V(G_1)$ in order to obtain $G_1'$ by elements of $\Tilde{\mathcal{M}}_2$.

\subsection*{Statement $1$ implies statement $3$}
By way of contradiction, we suppose that $ G $ is a minimal counterexample to statement 3 with respect to $|V(G)|+|E(G)|$. 
By Theorem \ref{Thm: min counterexample}, $G$ is $3$-connected.
By Statement 2 of Corollary \ref{Coro-planarity-minor-free}, either $G_s$ is planar or $G_s$ is isomorphic to $K_5$. Consequently, $G$ is planar, since $G$ is an $r$-graph. Thus, $G$ has a $(t,r)$-PM by statement $1$, a contradiction.
	
\subsection*{Statement $1$ implies statement $4$} 

Let $G$ be an $r$-graph with $cr(G_s) \leq 1$. 
Choose a drawing of $G_s$ with minimum number of crossing edges. 
Then there are at most two edges which cross in precisely one point. 
We will proceed by induction on $c_G$, which is defined as follows. 
If $G$ is planar, then $c_G = 0$. If $cr(G_s) = 1$, then choose among all drawings (with crossing number 1) a pair of crossing edges, say $xy$ and $uv$, such that $\mu_G(xy) \mu_G(uv)$ is minimum
and set $c_G = \mu_G(xy) \mu_G(uv)$.

If $c_G=0$, then $G$ is planar. Thus, $G$ has a $(t,r)$-PM by statement 1. 

If $c_G > 0$, then $cr(G_s) = 1$  and $\mu_G(x,y)\mu_G(u,v) \not = 0$. 
If there is a non-trivial tight edge-cut	
$\partial_G(X)$ and $E_G(x,y) \cup E_G(u,v) \subseteq \partial_G(X)$, then 
the two $r$-graphs obtained from $G$ by contracting $X$ and $V(G) \setminus X$, respectively, are planar. Thus, each of them has a $(t,r)$-PM.
These $(t,r)$-PMs can be combined to a $(t,r)$-PM of $G$. Therefore, for any $X\subset V(G)$ of odd cardinality, if $E_G(x,y) \cup E_G(u,v)\subseteq \partial_G(X)$, then $\vert\partial_G(X) \vert\geq r+2$.
		Let $e_{xy} \in E_G(x,y)$ and $e_{uv} \in E_G(u,v)$ and $G'$ be the graph
		that is obtained from $G - \{e_{xy},e_{uv}\}$ by adding edges $f = xu$ and 
  $f' = yv$. Then $G'$ is an $r$-graph, $cr(G_s') \leq 1$, and 
		$\mu_{G'}(x,y)\mu_{G'}(u,v) < \mu_G(x,y)\mu_G(u,v)$. Therefore,
		$G'$ has a $(t,r)$-PM $\ca M'$ by induction hypothesis. 
For every $e \in E(G')$ let $\ca M'(e)$ be the multiset of
the $t$ perfect matchings of $\ca M'$ which contain $e$.
If $M \in \ca M'$ contains $f$ and $f'$, then 
$M \in \ca M'(f) \cap \ca M'(f')$ and  
$M$ induces a perfect matching on $G$ which contains $e_{xy}$ and $e_{uv}$.
Thus, if $\ca M'(f) = \ca M'(f')$, then 
$\ca M'$ induces a $(t,r)$-PM $\ca M$ of $G$ and the statement is proved.
Hence, we assume $\ca M'(f) \not = \ca M'(f')$. Let $M_1, \dots, M_l$ be the elements
of $\ca M'(f)$ with $M_i \not \in \ca M'(f')$ and 
let $M_1', \dots, M_{l'}'$ be the elements
of $\ca M'(f')$ with $M_j' \not \in \ca M'(f)$. 
Note that $l = l'$. 
For $i \in \{1, \dots, l\}$,
consider the pairs $(M_i, M_i')$, which we call conflicting pairs.
Let $G'' = G' - \{f,f'\}$ ($=G- \{e_{xy},e_{uv}\}$)
and let $\ca M''$ be the restriction of $\ca M'$ to $G''$. 
Let $(M,M')$ be a conflicting pair.  
%Every vertex of $G''$ except $x,y,u,v$ is incident with an
%edge of $M$ and with an edge of $M'$. 
Every $(M,M')$-Kempe-chain in $G''$ contains no or two vertices of $x,y,u,v$. 
We call $(M,M')$ $G$-{\em extendable} if the restrictions of $M$ and $M'$ on $G''$ can be modified by Kempe-switching such that they extend to
two perfect matchings in $G$ (by adding $e_{xy}$ and $e_{uv}$).
%one of which contains
%$e_{xy}$ and not $e_{uv}$ and the other contains $e_{uv}$ and not $e_{xy}$.
Since every perfect matching in $\ca M'(f) \cap \ca M'(f')$ induces a 
perfect matching in $G$, which contains $e_{xy}$ and $e_{uv}$, it remains to prove that every conflicting pair is $G$-extendable.    
Let $(M_1,M_2)$ be a conflicting pair.
Consider the $(M_1,M_2)$-Kempe-chain $P_1$ (in $G''$) starting at $x$ and note that for $i \in \{1,2\}$,
\[(*) \hspace{.5cm} 
 E_{G''}(x,y) \cap M_i = \emptyset \text{ and } E_{G''}(u,v) \cap M_i = \emptyset. \]
If $P_1$ ends in $y$, then there is an $(M_1,M_2)$-Kempe-chain $P_2$ with end vertices $u$ and $v$. 
Since $G'' - (E_{G''}(x,y) \cup E_{G''}(u,v))$ is planar it follows with $(*)$ that $P_1$ and $P_2$ intersect, a contradiction. Consequently, $P_1$ ends in $u$ or in $v$. In both cases it follows 
that $(M_1,M_2)$ is $G$-extendable and the claim is proved. Hence, the $(t,r)$-PM $\ca M'$ on $G'$ induces a $(t,r)$-PM $\ca M$ on $G$ and the statement is proved. 
 
 \subsection*{Acknowledgements}

 \noindent  Yulai Ma and  Isaak H. Wolf were partially supported by Deutsche Forschungsgemeinschaft(DFG, German Research Foundation)-445863039. Junxue Zhang was partially supported by the China Postdoctoral Science Foundation (No. 2024M764113).

	\bibliography{Lit_reg_graph}{}
	\addcontentsline{toc}{section}{References}
	\bibliographystyle{abbrv}
	
\end{document}